\documentclass[12pt,amsfonts]{article}
\usepackage[margin=1in]{geometry}  
\usepackage{graphicx}              

\usepackage{amsmath, amssymb, amsthm, epsfig}               
\usepackage{enumerate}
\usepackage{amsfonts}              
\usepackage{amsthm}                
\usepackage{amsthm}
\theoremstyle{plain}
\usepackage{color}
\def\nd{\noindent}
\def\thend{\rule{3mm}{3mm}}

\newtheorem{claim}{Claim}[section]
\newtheorem{theorem}{Theorem}[section]

\newtheorem{proposition}{Proposition}[section]
\newtheorem{lemma}{Lemma}[section]
\newtheorem{rmk}{Remark}[section]

\newtheorem{cor}{Corollary}[section]
\newtheorem*{theorem*}{Theorem}

\numberwithin{equation}{section}

\begin{document}
\title{  A Hardy-Littlewood-Sobolev type inequality for variable exponents and applications to quasilinear Choquard equations involving variable exponent}
\author{Claudianor O. Alves\footnote{C.O. Alves was partially supported by CNPq/Brazil  301807/2013-2.} \,\, and \,\, Leandro da S. Tavares}  

\maketitle

\begin{abstract}
In this work we have proved a Hardy-Littlewood-Sobolev  inequality for variable exponents. After we use this inequality together with the variational method to establish the existence of solution for a class of Choquard equations involving the $p(x)$-Laplacian operator.  
\end{abstract}




\section{Introduction}

The stationary Choquard equation
\begin{equation}\label{Nonlocal.S1}
 -\Delta u +V(x)u =\Big(\int_{\mathbb{R}^N} \frac{|u|^p}{|x-y|^\lambda}\Big)|u|^{p-2}u  \quad \mbox{in} \quad \mathbb{R}^N
\end{equation}
where $N\geq3$,  $0<\lambda<N$, arises in many interesting physical situations in quantum theory and plays particularly an important
role in the theory of Bose-Einstein condensation where it accounts for the finite-range many-body interactions. For $N=3$, $p=2$ and $\lambda=1$, it was investigated by Pekar in \cite{P1} to study the quantum theory of a polaron at rest. In \cite{Li1}, Choquard applied it as approximation to Hartree-Fock theory of one-component plasma. This equation was also proposed by Penrose in \cite{MPT} as a model of selfgravitating matter and is known in that context as the Schr\"odinger-Newton equation. 

Motivated by these facts, at the last years a lot of articles have studied the existence and multiplicity of solutions for some problems  associated with (\ref{Nonlocal.S1}), see for example Ackermann \cite{AC}, Alves \& Yang \cite{AY1,AY2}, Cingolani, Secchi \& Squassina \cite{CCS}, Gao \& Yang \cite{GY}, Lions \cite{Ls}, Ma \& Zhao
\cite{ML},  Moroz \& Van Schaftingen \cite{MS2,MS3,MS4} and their references. 

In all the above mentioned papers the authors have used variational methods to show the existence of solution. This method works well thanks to a Hardy-Littlewood-Sobolev type inequality \cite{LL} that has the following statement 
\begin{proposition}[Hardy-Littlewood-Sobolev inequality]\label{HLS}
Let $t,r>1$ and $0<\lambda<N$ with $1/t+\lambda/N+1/r=2$, $f\in L^{t}(\mathbb{R}^N)$ and $h\in L^{r}(\mathbb{R}^N)$. There exists a sharp constant $C(t,N,\mu,r)$, independent of $f,h$, such that
\begin{equation}\label{HLS1}
\left|\int_{\mathbb{R}^{N}}\int_{\mathbb{R}^{N}}\frac{f(x)h(y)}{|x-y|^{\lambda}} dx dy \right|\leq C(t,N,\mu,r) \|f\|_{L^{t} (\mathbb{R}^N)}\|h\|_{L^{r}(\mathbb{R}^N)}.
\end{equation}
\end{proposition}

Motivated by the above papers, we intend to study the existence of solution for the following class of quasilinear problem 
\begin{equation} \label{Choquard-eq}
\left \{
\begin{array}{rclcl}
&-\Delta_{p(x)} u  + V(x) |u|^{p(x)-2}u = \left( \displaystyle\int_{\mathbb{R}^N} \frac{F(x,u(x))}{|x-y|^{\lambda(x,y)}}\right)f(y,u(y)) \ \mbox{in} \  \mathbb{R}^N, \\
&\\
&u \in W^{1,p(x)}(\mathbb{R}^N),
\end{array}
\right.
\end{equation}
where $V,p:\mathbb{R}^N \to \mathbb{R}$, $\lambda:\mathbb{R}^N \times \mathbb{R}^N \to \mathbb{R}$ and $f:\mathbb{R}^N \times \mathbb{R} \to \mathbb{R}$ are continuous functions, $F(x,t)$ is the primitive of $f(x,t)$, that is, 
$$
F(x,t)=\int_{0}^{t}f(x,s)\,ds
$$ 
and $\Delta_{p(x)}$ denotes the $p(x)$-Laplacian given by 
$$
\Delta_{p(x)}u=div(|\nabla u|^{p(x)-2}\nabla u).
$$ 

Our intention is showing that the variational method can also be used to prove the existence of solution for (\ref{Choquard-eq}). One of the main difficulties  is to show that the energy functional associated with (\ref{Choquard-eq}) given by 
$$
J(u)=\int_{\mathbb{R}^N}\frac{1}{p(x)}\left( |\nabla u|^{p(x)}+V(x)|u|^{p(x)}\right) -\frac{1}{2} \int_{\mathbb{R}^N} \int_{\mathbb{R}^N} \frac{F(x,u(x)) F(y,u(y))}{|x-y|^{\lambda(x,y)}}  dx dy
$$
is well defined and belongs to $C^{1}(W^{1,p(x)}(\mathbb{R}^N),\mathbb{R})$. In fact the main difficulty is to prove that
the functional $\Psi:W^{1,p(x)}(\mathbb{R}^N) \to \mathbb{R}$ given by
\begin{equation}\label{nonlocal-func}
\Psi(u)= \frac{1}{2} \int_{\mathbb{R}^N} \int_{\mathbb{R}^N} \frac{F(x,u(x))F(y,u(y))}{|x-y|^{\lambda(x,y)}} dx dy 
\end{equation}
belongs to $C^{1}(W^{1,p(x)}(\mathbb{R}^N), \mathbb{R})$ with
$$ 
\Psi'(u)v= \int_{\mathbb{R}^N} \int_{\mathbb{R}^N} \frac{F(x,u(x)f(y,u(y)) v(y))}{|x-y|^{\lambda(x,y)}} dxdy, \quad \forall u,v \in W^{1,p(x)}(\mathbb{R}^N).
$$

Here, we overcome this difficulty by establishing a Hardy-Littlewood-Sobolev type inequality for variable exponents, because this inequality is crucial in the proof that $\Psi \in C^{1}(W^{1,p(x)}(\mathbb{R}^N), \mathbb{R})$.

The $p(x)$-Laplacian operator possesses more complicated properties than
the $p$-Laplacian. For instance, it is inhomogeneous and in general, it has
no first eigenvalue, that is, the infimum of the eigenvalues of $p(x)$%
-Laplacian equals $0$ (see \cite{FZZ2}). Thus, transposing the results
obtained with the $p-$Laplacian to  problems with the $p(x)$-Laplacian
operator is not an easy task. The study of these problems are often very
complicated and require relevant topics of nonlinear functional analysis,
especially the theory of variable exponent Lebesgue and Sobolev spaces (see,
e.g., \cite{dhhr} and its abundant reference).

Partial differential equations involving the $p(x)$-Laplacian arise, for
instance, as a mathematical model for problems involving electrorheological
fluids and image restorations, see \cite{Acerbi1,Acerbi2,Antontsev,Chen,CLions,Ru}. This explains the intense
research on this subject in the last decades, see for example 
\cite{Alves2,Alves3,AlvesBarreiro,AlvesFerreira,AlvesFerreira1,AlvesSouto,Fan1,Fan,FanZhao0,FZ1,FZZ, FanShenZhao,BSS,BSS1,FuZa,MR,R*,RR*} and their references.

The plan of the paper is as follows: In Section 2 we recall some facts involving the variable exponent Sobolev space and prove the Hardy-Littlewood-Sobolev  inequality for variable exponents. In Section 3 we show that $\Psi$ is $C^{1}$ while in Section 4 we study the existence of solution of (\ref{Choquard-eq}) by assuming some conditions on $V(x)$ and $f(x,t)$.

\section{Variable exponent Sobolev space}
In this section we recall some results on variable exponent Sobolev spaces. The reader is referred to \cite{FZ,FanShenZhao}  and their references for more details.

In the sequel, we set
$$ 
C^{+}(\mathbb{R}^N):= \{ h \in C(\mathbb{R}^N)\,:\, 1<h^{-} \leq h^{+} < +\infty \}
$$
where
$$
h^{+}:= \sup_{x \in \mathbb{R}^N}h(x) \quad \mbox{and} \quad h^{-}:= \inf_{x \in \mathbb{R}^N}h(x). 
$$
For $p \in C^{+}(\mathbb{R}^N)$, let us consider the Lebesgue space
$$
L^{p(x)}(\mathbb{R}^N)= \left\{ u:\mathbb{R}^N \to \mathbb{R}; u \ \text{is a measurable and} \int_{\mathbb{R}^N}|u(x)|^{p(x)} \ dx < + \infty \right\}, 
$$
which becomes a Banach space when endowed with the Luxemburg norm
$$ 
\| u \|_{L^{p(x)}(\mathbb{R}^N)}=  \inf \left\{ \alpha >0; \int_{\mathbb{R}^N} \left|\frac{u(x)}{\alpha} \right|^{p(x)}  \ dx \leq 1\right\}.
$$

\begin{proposition}\label{norm-prop}
The functional $\rho_{p}: L^{p(x)}(\mathbb{R}^N) \rightarrow \mathbb{R}$ defined by
$$ 
\rho_{p}(u)=\int_{\mathbb{R}^N} |u(x)|^{p(x)} dx
$$
has the following properties:
\begin{enumerate}[(i)]
\item  $\| u\|_{L^{p(x)}(\mathbb{R}^N)} < 1 (=1;>1) \Longleftrightarrow \rho_{p}(u) < 1 (=1;>1). $
\item $\| u\|_{L^{p(x)}(\mathbb{R}^N)} > 1 \Longrightarrow \| u\|^{p^{-}}_{L^{p(x)}(\mathbb{R}^N)} \leq \rho_{p}(u)  \leq\| u\|^{p^{+}}_{L^{p(x)}(\mathbb{R}^N)}$. \\
$\| u\|_{L^{p(x)}(\mathbb{R}^N)} < 1 \Longrightarrow \| u\|^{p^{+}}_{L^{p(x)}(\mathbb{R}^N)} \leq \rho_{p}(u)  \leq \| u\|^{p^{-}}_{L^{p(x)}(\mathbb{R}^N)}$. 
\item $\| u_n\|_{L^{p(x)}(\mathbb{R}^N)}(\mathbb{R}^N) \rightarrow 0 \Longleftrightarrow \rho_{p}(u_n) \rightarrow 0 \ ; \ \| u_n\|_{L^{p(x)}(\mathbb{R}^N)} \rightarrow \infty \Longleftrightarrow \rho_{p}(u_n) \rightarrow \infty$.
\end{enumerate}
\end{proposition}
For $p \in C^{+}(\mathbb{R}^N)$, let $p': \mathbb{R}^N \rightarrow \mathbb{R}$ such that
$$ \frac{1}{p(x)} + \frac{1}{p'(x)} =1, \ \text{a.e} \ x \in \mathbb{R}^N.$$
We have the following generalized H\"{o}lder inequality.
\begin{proposition}[\cite{Musielak}] For any $u \in L^{p(x)}(\mathbb{R}^N)$ and $v \in  L^{p'(x)}(\mathbb{R}^N)$,
$$
\left|\int_{\mathbb{R}^N} u(x)v(x) dx\right| \leq 2 \| u\|_{L^{p(x)}(\mathbb{R}^N)} \| v\|_{L^{p'(x)}(\mathbb{R}^N)}.
$$
	
\end{proposition}

The Banach space $W^{1,p(x)}(\mathbb{R}^N)$ is defined as 
$$ 
W^{1,p(x)}(\mathbb{R}^N):= \left\{u \in L^{p(x)}(\mathbb{R}^N); |\nabla u| \in L^{p(x)}(\mathbb{R}^N) \right\}
$$
equipped with the norm
$$ 
\| u \|_{W^{1,p(x)}(\mathbb{R}^N)}:= \| u\|_{L^{p(x)}(\mathbb{R}^N)} + \| \nabla u\|_{L^{p(x)}(\mathbb{R}^N)}.
$$

In what follows, we denote by $h\ll g$ provided $\inf\{ g(x)-h(x); x \in \mathbb{R}^N\}>0$ where $h$ and $g$ are continuous functions. The following embeddings will be used frequently  in this work.

\begin{proposition}[\cite{FanShenZhao}]\label{cpt-emb}
Let $p: \mathbb{R}^N \rightarrow \mathbb{R}$ be a Lipschitz continuous function with $1<p^{-}\leq p^{+} < N$ and $s \in C^{+}(\mathbb{R}^N)$. 

\begin{enumerate}[(i)]
\item If $p \leq s \leq p^{\star}$, then there is a continuous embedding $W^{1,p(x)}(\mathbb{R}^N) \hookrightarrow L^{s(x)}(\mathbb{R}^N)$.	
\item If $p \leq s \ll p^{\star}$ then there is a compact embedding $W^{1,p(x)}(\mathbb{R}^N) \hookrightarrow L^{s(x)}_{\rm loc}(\mathbb{R}^N)$,	
\end{enumerate}
where  $p^{\star}(x):= Np(x)/(N-p(x))$ for all  $x \in \mathbb{R}^N.$	
\end{proposition}
We also need of the following Lions' Lemma for variable exponent found in \cite{Fan-Zhao-Zhao}. For $r>0$ and $ y\in\mathbb{R}^N$ we denote by $B_r(y)$ the open ball in $\mathbb{R}^N$ with center $y$ and radius $r$.

\begin{lemma}\label{lions-variable}Let $p: \mathbb{R}^N \rightarrow \mathbb{R}$ be a Lipschitz continuous function with $1<p^{-}\leq p^{+} < N$. If $(u_n)$ is a bounded sequence in $W^{1,p(x)}(\mathbb{R}^N)$ such that
	$$ \lim_{n \rightarrow + \infty} \sup_{y \in \mathbb{R}^N} \int_{B_r(y)}|u_n(x)|^{p(x)} dx = 0$$
for some $r>0$, then $u_n \rightarrow 0$ in $L^{q(x)}(\mathbb{R}^N)$ for any $q \in C^{+}(\mathbb{R}^N)$ satisfying $p \ll q \ll p^{\star}.$
\end{lemma}

Next, we prove Hardy-Littlewood-Sobolev type inequality for variable exponents.

\begin{proposition}[Hardy-Littlewood-Sobolev type inequality for variable exponents] \label{HSL-variable}  Let $p,q \in C^{+}(\mathbb{R}^N)$, $h \in L^{p^+}(\mathbb{R}^N)\cap L^{p^-}(\mathbb{R}^N)$, $g \in L^{q^+}(\mathbb{R}^N)\cap L^{q^-}(\mathbb{R}^N)$ and $\lambda: \mathbb{R}^N \times \mathbb{R}^N \rightarrow \mathbb{R}$ be a continuous function such that $0<\lambda^{-}\leq \lambda^{+} < N$ and 
$$ 
\frac{1}{p(x)}  +\frac{\lambda(x,y)}{N} + \frac{1}{q(y)} = 2, \quad \forall  x,y \in \mathbb{R}^N.
$$
Then,
$$ 
\left| \int_{\mathbb{R}^N}\int_{\mathbb{R}^N} \frac{h(x)g(y)}{|x-y|^{\lambda(x,y)}} dx dy \right|\leq C(\| h\|_{L^{p^+}(\mathbb{R}^N)} \| g\|_{L^{q^+}(\mathbb{R}^N)} + \| h\|_{L^{p^-}(\mathbb{R}^N)} \| g\|_{L^{q^-}(\mathbb{R}^N)}) 
$$
where $C>0$ is a constant that does not depend on $h$ and $g.$
\end{proposition}
\begin{proof} First of all, note that 
$$
	\lambda(x,y) = 2N \left( 1 - \frac{1}{2p(x)} - \frac{1}{2q(y)}\right) \leq 2N \left( 1 - \frac{1}{2p^+} - \frac{1}{2q^{+}}\right), \quad \forall x,y \in \mathbb{R}^N.
$$
Therefore, 
$$
\lambda^{+}=\sup_{x,y \in \mathbb{R}^N}\lambda(x,y)\leq 2N \left( 1 - \frac{1}{2p^+} - \frac{1}{2q^{+}} \right).
$$
Now, if $(x_n), (y_n) \subset \mathbb{R}^N$ are sequences satisfying
$$
p(x_n) \to p^+ \quad \mbox{and} \quad q(y_n) \to q^+ 
$$
we see that 
$$
\lambda(x_n,y_n) \to 2N \left( 1 - \frac{1}{2p^+} - \frac{1}{2q^{+}}\right)
$$
from where it follows 
$$
\lambda^{+}=2N \left( 1 - \frac{1}{2p^+} - \frac{1}{2q^{+}} \right)
$$
or equivalently
\begin{equation}\label{HSL2'}
\frac{1}{p^{+}} + \frac{\lambda^+}{N} + \frac{1}{q^{+}} =2.
\end{equation}
Likewise
\begin{equation}\label{HSL2}
\frac{1}{p^{-}} + \frac{\lambda^{-}}{N} + \frac{1}{q^{-}} =2.
\end{equation}
Since
$$
\frac{1}{|x-y|^{\lambda(x,y)}} \leq \frac{1}{|x-y|^{\lambda^+}} +
\frac{1}{|x-y|^{\lambda^-}} \quad \forall x,y \in \mathbb{R}^N,
$$ 
we derive that 
$$
\left| \int_{\mathbb{R}^N}\int_{\mathbb{R}^N} \frac{h(x)g(y)}{|x-y|^{\lambda(x,y)}} dx dy \right| \leq  \int_{\mathbb{R}^N}\int_{\mathbb{R}^N} \frac{|h(x)||g(y)|}{|x-y|^{\lambda^{-}}}dx dy +  \int_{\mathbb{R}^N}\int_{\mathbb{R}^N} \frac{|h(x)||g(y)|}{|x-y|^{\lambda^{+}}} dx dy
$$
Gathering (\ref{HSL2'}), (\ref{HSL2}) and  Proposition \ref{HLS1} we get 
$$
\int_{\mathbb{R}^N}\int_{\mathbb{R}^N} \frac{|h(x)||g(y)|}{|x-y|^{\lambda^{-}}}dx dy \leq C\| h\|_{L^{p^+}(\mathbb{R}^N)} \| g\|_{L^{q^+}(\mathbb{R}^N)}
$$
and
$$
\int_{\mathbb{R}^N}\int_{\mathbb{R}^N} \frac{|h(x)||g(y)|}{|x-y|^{\lambda^{+}}} dx dy
\leq C \| h\|_{L^{p^-}(\mathbb{R}^N)} \| g\|_{L^{q^-}(\mathbb{R}^N)}.
$$
From the last two inequalities,  
$$
\left| \int_{\mathbb{R}^N}\int_{\mathbb{R}^N} \frac{h(x)g(y)}{|x-y|^{\lambda(x,y)}} dx dy \right|\leq  C(\| h\|_{L^{p^+}(\mathbb{R}^N)} \| g\|_{L^{q^+}(\mathbb{R}^N)} + \| h\|_{L^{p^-}(\mathbb{R}^N)} \| g\|_{L^{q^-}(\mathbb{R}^N)}),
$$
showing the inequality. 
\end{proof}
In this  work we will also consider that $p: \mathbb{R}^N \rightarrow \mathbb{R}$ is a Lipschitz function with $p \in C^{+}(\mathbb{R}^N).$ 
The next corollary is a key point in our arguments.
\begin{cor}\label{corollary}  Let $q \in C^{+}(\mathbb{R}^N)$ and $\lambda: \mathbb{R}^{N} \times \mathbb{R}^{N} \to \mathbb{R}$ be a function satisfying 
\begin{equation}\label{symetry}
\frac{1}{q(x)} + \frac{\lambda(x,y)}{N}+ \frac{1}{q(y)}=2, \quad \forall x,y \in \mathbb{R}^N.
\end{equation}
If $u \in W^{1,p(x)}(\mathbb{R}^N)$ and $r \in \mathcal{M}$ where 
\begin{equation}\label{growth1}
\mathcal{M}=\left\{ r \in C^{+}(\mathbb{R}^N)\,:\,p(x)\leq r(x) q^{-} \leq r(x)q^{+} \leq p^{\star}(x), \quad \forall x \in \mathbb{R}^N  \right\},
\end{equation}
then $U(x)=|u(x)|^{r(x)} \in L^{q^-}(\mathbb{R}^N) \cap L^{q^+}(\mathbb{R}^N)$. Moreover,
$$
\begin{array}{l}
\displaystyle \left| \int_{\mathbb{R}^N}\int_{\mathbb{R}^N} \frac{|u(x)|^{r(x)}|u(y)|^{r(y)}}{|x-y|^{\lambda(x,y)}} dx dy \right| \leq  C(\| |u|^{r(x)}\|^{2}_{L^{q^{+}}(\mathbb{R}^N)} + \| |u|^{r(y)}\|^{2}_{L^{q^{-}}(\mathbb{R}^N)})
\end{array}
$$
where $C$ is a constant that does not depend on $u \in W^{1,p(x)}(\mathbb{R}^N)$.
\end{cor}
\begin{proof} Using Sobolev embedding (Proposition \ref{cpt-emb}), $ u \in L^{s(x)}(\mathbb{R}^N)$
for all  $s \in  C^{+}(\mathbb{R}^N)$ with  
$$
p(x) \leq s(x) \leq p^{*}(x), \quad \forall x \in \mathbb{R}^N.
$$
Thereby, $U(x)=|u(x)|^{r(x)} \in L^{q^+} \cap L^{q^-}(\mathbb{R}^N)$, because $r \in \mathcal{M}$. Now, we use the Proposition \ref{HSL-variable} with $p(x)=q(x)$ and $h(x)=g(x)=U(x)$ to obtain the desired result. 

\end{proof}

Before continuing our study, we would like point out some important properties of the function $\lambda(x,y)$ given in (\ref{symetry}).
\begin{rmk} \label{RMK1} \mbox{} \\
\noindent $i)$ \, The function $\lambda$ is symmetric, that is,
$$
\lambda(x,y)=\lambda(y,x), \quad \forall x,y \in \mathbb{R}^N.
$$
\noindent ii) \, If $q$ is $\mathbb{Z}^N$-periodic, that is,
$$
q(x+y)=q(x), \quad \forall x \in \mathbb{R}^N \quad \mbox{and} \quad \forall y \in \mathbb{Z}^N,
$$
then $\lambda$ is $\mathbb{Z}^N \times \mathbb{Z}^N$-periodic, that is,
$$
\lambda (x+z,y+w)=\lambda (x,y), \quad \forall x,y \in \mathbb{R}^N \quad \mbox{and} \quad \forall z,w \in \mathbb{Z}^N.
$$
\noindent iii) \, If $q$ is radial, that is,
$$
q(x)=q(|x|), \quad \forall x \in \mathbb{R}^N
$$
then 
$$
\lambda (x,y)=\lambda (|x|,|y|), \quad \forall x,y \in \mathbb{R}^N.
$$
\end{rmk}	

Part $i)$ in Remark \ref{RMK1} will be crucial in the proof of the differentiability of  functional $\Psi$.

\section{Differentiability of the functional $\Psi$.}

In this section, we will study the differentiability of functional $\Psi$ given in (\ref{nonlocal-func}). To this end, we must assume some conditions on $f$. First of all, we fix $q \in C^{+}(\mathbb{R}^N)$ and $\lambda: \mathbb{R}^N \times \mathbb{R}^N \to \mathbb{R}$ satisfying (\ref{symetry}), that is, 
$$
\frac{1}{q(x)} + \frac{\lambda(x,y)}{N}+ \frac{1}{q(y)}=2, \quad \forall x,y \in \mathbb{R}^N.
$$

The function $f: \mathbb{R}^N \times \mathbb{R} \rightarrow \mathbb{R}$ is a  continuous function verifying the following growth condition
$$
|f(x,t)| \leq C_1 (|t|^{r(x)-1} + |t|^{s(x)-1}), \quad \forall  (x,t) \in \mathbb{R}^N \times \mathbb{R} \eqno({f_1})
$$
where $C_1>0$ and $r,s \in \mathcal{M}$ given by (\ref{growth1}).

Note that the function $F(x,t):= \int_{0}^{t}f(x,s) ds $ is continuous and 
$$
|F(x,t)| \leq C_2 (|t|^{r(x)} + |t|^{s(x)}), \quad \forall (x,t ) \in  \mathbb{R}^N \times \mathbb{R} \eqno({F})
$$
for some positive constant $C_2.$

In the proof of the differentiability of $\Psi$ we will use the lemma below whose  proof we omit because it is very simple. 

\begin{lemma} \label{DERIVADA} Let $E$ be a normed vectorial space and $J:E \to \mathbb{R}$ be a functional verifying the following properties:\\
\noindent $i)$ \, The Fr\'echet derivative $\frac{\partial J(u)}{\partial v}:= \lim\limits_{t \rightarrow 0} \frac{J(u + tv) - J(u)}{t}$ exists for all $u, v \in E,$ \\
\noindent $ii)$ \, For each $u \in E$, $\frac{\partial J(u)}{\partial(.)} \in E'$, that is, the application $v \longmapsto \frac{\partial J(u)}{\partial v}$ is a continuous linear functional, \\
\noindent $iii)$ 
$$
u_n \to u \quad \mbox{in} \quad E \Longrightarrow \frac{\partial J(u_n)}{\partial (.)} \to \frac{\partial J(u)}{\partial (.)} \quad \mbox{in} \quad E' 
$$
that is,
$$
u_n \to u \quad \mbox{in} \quad E \Longrightarrow \sup_{\| v\|\leq 1} \left| \frac{\partial J(u_n)}{\partial v} - \frac{\partial J(u)}{\partial v} \right| \rightarrow 0.  
$$  	
Then, $J \in C^{1}(E, \mathbb{R})$ and 
$$
J'(u)v=\frac{\partial J(u)}{\partial v}, \quad \forall u,v \in E.
$$	
\end{lemma}	

After the above establishments we are ready to prove the differentiability of functional $\Psi$ given by \eqref{nonlocal-func}. 
\begin{lemma}\label{C1}
The functional $\Psi$ given in \eqref{nonlocal-func} is well defined and  belongs to $C^1(W^{1,p(x)}(\mathbb{R}^N),\mathbb{R})$ with 
$$ 
\Psi'(u)v= \int_{\mathbb{R}^N} \int_{\mathbb{R}^N} \frac{F(x,u(x)f(y,u(y)) v(y))}{|x-y|^{\lambda(x,y)}} dxdy, 
$$
for all $u,v \in W^{1,p(x)}(\mathbb{R}^N).$
\end{lemma}

\begin{proof} From $(f_1), (F)$ and Proposition \ref{HSL-variable} it follows that $\Psi$ is well defined. In the sequel, we will show that $\Psi$ satisfies the assumptions of Lemma  \ref{DERIVADA}. To this end, we will divide the proof into three steps: \\
	
\noindent \textbf{Step 1:} Existence of the Fr\'echet derivative:	\\

Let $u,v \in W^{1,p(x)}(\mathbb{R}^N)$ and $t \in [-1,1]$. Note that
\begin{equation}\label{integrand}
\frac{\Psi(u+tv) - \Psi(u)}{t} = \frac{1}{2} \int_{\mathbb{R}^N} \int_{\mathbb{R}^N} \frac{F(x,u(x) + tv(x))F(y,u(y) + tv(y)) - F(x,u(x))F(y,u(y))}{t |x-y|^{\lambda(x,y)}} dx dy. 
\end{equation}
Denoting by $I$ the integrand in \eqref{integrand}, we have
$$
\begin{array}{l}
\displaystyle I= \frac{F(x,u(x)+tv(x))(F(y,u(y)+tv(y)) - F(y,u(y))}{t} \\
\mbox{}\\
\displaystyle + \frac{F(y,u(y))(F(x,u(x)+tv(x)) - F(x,u(x)))}{t}.
\end{array}
$$
By the Mean Value Theorem there exist $\theta(x,t), \eta(y,t)\in [0,1]$, such that
$$
F(y,u(y) + tv(y)) - F(y,u(y)) = f(y,u(y) + \eta(y,t)tv(y))v(y)t 
$$
and
$$
F(x,u(x) + tv(x)) - F(x,u(x)) =  f(x,u(x) + \theta(x,t)tv(x))v(x)t.
$$
The relation  \eqref{integrand} allows us to estimate
$$
\left|\frac{\Psi(u+tv) - \Psi(u)}{t} - \int_{\mathbb{R}^N} \int_{\mathbb{R}^N} \frac{F(x,u(x)f(y,u(y)) v(y))}{|x-y|^{\lambda(x,y)}} dxdy \right|  \leq |B^{t}_{1}|+|B^{t}_{2}| 
$$
where
\begin{align*}B^{t}_{1} &: =  \frac{1}{2} \int_{\mathbb{R}^N} \int_{\mathbb{R}^N} \frac{F(x,u(x)+tv(x)) f(y,u(y) + \eta(y,t)tv(y))v(y) - F(x,u(x)f(y,u(y))v(y))}{|x-y|^{\lambda(x,y)}} dx dy
\end{align*}
and
\begin{align*}B^{t}_{2} &: =  \frac{1}{2} \int_{\mathbb{R}^N} \int_{\mathbb{R}^N} \frac{F(y,u(y))f(x,u(x) + \theta(x,t)tv(x))v(x)}{|x-y|^{\lambda(x,y)}} dx dy - \frac{1}{2} \int_{\mathbb{R}^N} \int_{\mathbb{R}^N} \frac{F(x,u(x))f(y,u(y))v(y)}{|x-y|^{\lambda(x,y)}} dx dy
\\
&
\end{align*}
By Remark \ref{RMK1}, the function $\lambda(x,y)$ is symmetric, that is,
$$
\lambda(x,y)=\lambda(y,x), \quad \forall x,y \in \mathbb{R}^N.
$$
Such property combined with Fubini's Theorem implies that
\begin{align*}
\int_{\mathbb{R}^N} \int_{\mathbb{R}^N} \frac{F(x,u(x))f(y,u(y))v(y)}{|x-y|^{\lambda(x,y)}} dx dy & = \int_{\mathbb{R}^N} \int_{\mathbb{R}^N} \frac{F(y,u(y))f(x,u(x))v(x)}{|x-y|^{\lambda(y,x)}} dy dx \\
&=\int_{\mathbb{R}^N} \int_{\mathbb{R}^N}\frac{F(y,u(y))f(x,u(x))v(x)}{|x-y|^{\lambda(y,x)}} dx dy  \\ 
&=\int_{\mathbb{R}^N}\int_{\mathbb{R}^N}\frac{F(y,u(y))f(x,u(x))v(x)}{|x-y|^{\lambda(x,y)}} dx dy.
\end{align*}
Therefore, $B^{t}_{2}$ can be rewritten as
$$
B^{t}_{2}:= \frac{1}{2} \int_{\mathbb{R}^N} \int_{\mathbb{R}^N} \frac{F(y,u(y))f(x,u(x) + \theta(x,t)tv(x))v(x) - F(y,u(y)f(x,u(x))v(x))}{|x-y|^{\lambda(x,y)}} dxdy. 
$$
Then, by Proposition \ref{HSL-variable} 
\begin{align*}
|B^{t}_{2}| &\leq C\| F(.,u)\|_{L^{q^{+}}(\mathbb{R}^N)} \| f(., u + \theta(.,t)v)v-f(.,u)v\|_{L^{q^{+}}(\mathbb{R}^N)}\\
&+C\| F(.,u)\|_{L^{q^{-}}(\mathbb{R}^N)} \| f(., u + \theta(.,t)v)v-f(.,u)v\|_{L^{q^{-}}(\mathbb{R}^N)}.
\end{align*}
Since $\theta(x,t) \in [0,1]$ and $t \in [-1,1]$, the condition $(f_1)$ guarantees that
\begin{align}\label{uniform}
|f(x, u(x) + \theta(t,x)tv(x))v(x) - f(x,u(x))v(x)|^{q^{+}} \leq C(|u(x)|^{q^{+}(r(x)-1)}|v(x)|^{q^{+}} + |v(x)|^{q^{+}r(x)}) \nonumber \\
+ C(|u(x)|^{q^{+}(s(x)-1)}|v(x)|^{q^{+}} + |u(x)|^{q^{+}s(x)} + |u(x)|^{q^{+}(r(x)-1)}|v(x)|^{q^{+}} + |u(x)|^{q^{+}(s(x)-1)}|v(x)|^{q^{+}}) 
\end{align}
The growth conditions \eqref{growth1} and Proposition  \ref{cpt-emb} ensure that the right side of the inequality \eqref{uniform} is an integrable function. Thus, the Lebesgue's Dominated Convergence Theorem  gives 
$$
\| f(., u + \theta(.,t)v)v-f(.,u)v\|_{L^{q^{+}}(\mathbb{R}^N)} \rightarrow 0 \ as \ t\rightarrow 0 .
$$
Likewise
$$
\| f(., u + \theta(.,t)tv)v-f(.,u)v\|_{L^{q^{-}}(\mathbb{R}^N)} \rightarrow 0 \ as \ t\rightarrow 0.
$$
The last limits imply that $B^{t}_{2} \rightarrow 0$ as  $t\rightarrow 0.$  Related to the $B_{1}^{t}$, we have the estimate below 
\begin{align*}
|B^{t}_{1}|&\leq \frac{1}{2} \int_{\mathbb{R}^N} \int_{\mathbb{R}^N} \frac{|F(x,u(x))| |f(y,u(y) + \eta(y,t)tv(y))v(y) - f(y,v(y))v(y)|}{|x-y|^{\lambda(x,y)}} dx dy \\
&+ \frac{1}{2} \int_{\mathbb{R}^N} \int_{\mathbb{R}^N} \frac{|f(y,u(y) + \eta(y,t)tv(y))v(y)| |F(x,u(x) + t v(x)) - F(x,u(x))|}{|x-y|^{\lambda(x,y)}} dx dy.
\end{align*}
Arguing as above,   
$$ \int_{\mathbb{R}^N} \int_{\mathbb{R}^N} \frac{|F(x,u(x))| |f(y,u(y) + \eta(y,t)tv(y))v(y) - f(y,v(y))v(y)|}{|x-y|^{\lambda(x,y)}} dx dy \rightarrow 0$$
as $t \rightarrow 0.$ On the other hand, the Lebesgue's Dominated Convergence Theorem also yields 
\begin{equation}\label{F1}
\| F(.,u + tv) - F(.,u)\|_{L^{q^{+}}(\mathbb{R}^N)} \rightarrow 0 \ \text{as} \ t\rightarrow 0
\end{equation}
and
\begin{equation}\label{F2}
\| F(.,u + tv) -F(.,u)\|_{L^{q^{-}}(\mathbb{R}^N)} \rightarrow 0 \ \text{as} \ t\rightarrow 0.
\end{equation}
As in \eqref{uniform}, the quantities $\| f(.,u + \eta(.,t)tv)v\|_{L^{q^{+}}(\mathbb{R}^N)}$ and  $\| f(.,u + \eta(.,t)tv)v\|_{L^{q^{-}}(\mathbb{R}^N)}$ are uniformly bounded by a constant that does not depend on $t \in [-1,1]$. Thus, the Proposition \ref{HSL-variable} combined with \eqref{F1} and \eqref{F2}  gives 
$$ 
\int_{\mathbb{R}^N} \int_{\mathbb{R}^N} \frac{|f(y,u(y) + \eta(y,t)tv(y))v(y)| |F(x,u(x)+tv(x)) - F(x,u(x))|}{|x-y|^{\lambda(x,y)}} dx dy \rightarrow 0
$$
as $t \rightarrow 0,$ and so, $ B^{t}_{1} \rightarrow 0$ as $t \rightarrow 0.$
From the above analysis, 
$$
 \lim_{t \rightarrow 0} \frac{\Psi(u+tv) - \Psi(u)}{t} =  \int_{\mathbb{R}^N} \int_{\mathbb{R}^N} \frac{F(x,u(x)f(y,u(y))v(y))}{|x-y|^{\lambda(x,y)}} dx dy, 
$$
showing the existence the existence of the Fr\'echet derivative $\frac{\partial \Psi(u)}{\partial v}.$ \\

\noindent \textbf{Step 2:} $\frac{\partial \Psi(u)}{\partial (.)} \in (W^{1,p(x)}(\mathbb{R}^N))'$ for all $u \in W^{1,p(x)}(\mathbb{R}^N)$ . \\

It is evident that $\frac{\partial \Psi(u)}{\partial v}$ is linear at $v$ for each $u$ fixed.  Next, we are going to show that 
$$
\left| \frac{\partial \Psi(u)}{\partial v}\right| \leq C_u \| v\|, \quad \forall v \in W^{1,p(x)}(\mathbb{R}^N),
$$ 
for some positive constant $C_u$ that does not depend on $v \in W^{1,p(x)}(\mathbb{R}^N)$. From $(f_1),(F)$ and Proposition \ref{HSL-variable}   
\begin{equation}\label{continuous1}
\begin{aligned}
\left|\int_{\mathbb{R}^N} \int_{\mathbb{R}^N} \frac{F(x,u(x))f(y,u(y))v(y)}{|x-y|^{\lambda(x,y)}} dx dy\right| &\leq C \| F(.,u)\|_{L^{q^{+}}(\mathbb{R}^N)}\| f(.,u)v\|_{L^{q^{+}}(\mathbb{R}^N)}   \\
& + C \| F(.,u)\|_{L^{q^{-}}(\mathbb{R}^N)}\| f(.,u)v\|_{L^{q^{-}}(\mathbb{R}^N)} .
\end{aligned}
\end{equation}
Suppose that $\|v\|_{W^{1,p(x)}(\mathbb{R}^N)} \leq 1$.  The continuous embeddings   ${W^{1,p(x)}(\mathbb{R}^N)} \hookrightarrow L^{p^{+}r(x)}(\mathbb{R}^N)$ and ${W^{1,p(x)}(\mathbb{R}^N)} \hookrightarrow L^{p^{+}s(x)}(\mathbb{R}^n)$ (see Proposition \ref{cpt-emb}) combined with H\"{o}lder inequality, $(f_1)$ and Proposition \ref{norm-prop} give
\begin{align*}
\int_{\mathbb{R}^N} |f(y,u(y))v(y)|^{q^{+}} dy & \leq C \| |u|^{q^{+}(r(.)-1)}\|_{L^{\frac{r(y)}{r(y)-1}}(\mathbb{R}^N)}  \| |v|^{q^{+}}\|_{L^{r(y)}(\mathbb{R}^N)}\\
&+ C \||u|^{q^{+}(s(.)-1)}\|_{L^{\frac{s(y)}{s(y)-1}}(\mathbb{R}^N)} \| |v|^{q^{+}}\|_{L^{s(x)}(\mathbb{R}^N)}
\end{align*} 
\begin{equation}\label{continuous2}
\begin{aligned}
&\leq C_u   \left(  \max  \left( \| v\|^{q^{+}}_{L^{q^{+}r(x)}(\mathbb{R}^N)}  , \| v\|^{\frac{q^{+}r^{-}}{r^{+}}}_{L^{q^{+}r(x)}(\mathbb{R}^N)}  \right) + \max \left( \| v\|^{\frac{q^{+}r^{+}}{r^{-}}}_{L^{q^{+}r(x)}(\mathbb{R}^N)}, \| v\|^{q^{+}}_{L^{q^{+}r(x)}(\mathbb{R}^N)}  \right)\right) \\
&+ C_u   \left(  \max  \left( \| v\|^{q^{+}}_{L^{q^{+}s(x)}(\mathbb{R}^N)}  , \| v\|^{\frac{q^{+}s^{-}}{s^{+}}}_{L^{q^{+}s(x)}(\mathbb{R}^N)}  \right) + \max \left( \| v\|^{\frac{q^{+}s^{+}}{s^{-}}}_{L^{q^{+}s(x)}(\mathbb{R}^N)}, \| v\|^{q^{+}}_{L^{q^{+}s(x)}(\mathbb{R}^N)}  \right)\right) \\
& \leq C_{u_1}
\end{aligned}
\end{equation}
where
\begin{align*}
C_{u_1}&:=K_1 \left( \max \left( \left( \int_{\mathbb{R}^N} |u(y)|^{q^{+}r(y)} dy \right)^{\frac{1}{\left(\frac{r}{r-1}\right)^+}}, \left( \int_{\mathbb{R}^N} |u(y)|^{q^{+}s(y)} dy \right)^{\frac{1}{\left(\frac{r}{r-1}\right)^-}}  \right)\right) \\
&+ K_1 \left( \max \left( \left( \int_{\mathbb{R}^N} |u(y)|^{q^{+}s(y)} dy \right)^{\frac{1}{\left(\frac{s}{s-1}\right)^+}}, \left( \int_{\mathbb{R}^N} |u(y)|^{q^{+}s(y)} dy \right)^{\frac{1}{\left(\frac{s}{s-1}\right)^+}}  \right)\right)
\end{align*}
and $K_1$ is a constant that does not depend on $u$ and $v.$ The previous argument also implies that, 
\begin{equation}\label{continuous3}
\| f(.,u)v\|_{L^{q^{-}}(\mathbb{R}^N)} \leq C_{u_2}, \ \text{for all} \ v \in W^{1,p(x)}(\mathbb{R}^N) \ \text{with} \ \|v \|_{W^{1,p(x)}(\mathbb{R}^N)} \leq 1
\end{equation}
where
\begin{align*}
C_{u_2}&:=K_2 \left( \max \left( \left( \int_{\mathbb{R}^N} |u(y)|^{q^{-}r(y)} dy \right)^{\frac{1}{\left( \frac{r}{r-1}\right)^+}}, \left( \int_{\mathbb{R}^N} |u(y)|^{q^{-}s(y)} dy \right)^{\frac{1}{\left( \frac{r}{r-1}\right)^-}}   \right)\right) \\
&+ K_2 \left( \max \left( \left( \int_{\mathbb{R}^N} |u(y)|^{q^{-}s(y)} dy \right)^{\frac{1}{\left( \frac{s}{s-1}\right)^+}}, \left( \int_{\mathbb{R}^N} |u(y)|^{q^{-}s(y)} dy \right)^{\frac{1}{\left( \frac{s}{s-1}\right)^+}}   \right)\right)
\end{align*}
with $K_2$ being a constant that does not depend on $u$ and $v.$
The inequalities \eqref{continuous1},\eqref{continuous2} and \eqref{continuous3} justify the Step 2. \\

\noindent \textbf{Step 3:} 
\begin{equation} \label{B0}
u_n \to  u \quad \mbox{in} \quad W^{1,p(x)}(\mathbb{R}^N) \Rightarrow \sup_{\| v\|_{W^{1,p(x)}(\mathbb{R}^N)}\leq 1} \left| \frac{\partial \Psi(u_n)}{\partial v} - \frac{\partial \Psi(u)}{\partial v} \right| \rightarrow 0. 
\end{equation}
Consider $v \in W^{1,p(x)}(\mathbb{R}^N)$ with $\| v\|_{W^{1,p(x)}(\mathbb{R}^N)} \leq 1$ and note that  
\begin{align*}
\left| \frac{\partial \Psi(u_n)}{\partial v} - \frac{\partial \Psi(u)}{\partial v} \right| & \leq \int_{\mathbb{R}^N} \int_{\mathbb{R}^N} \frac{|F(x,u_n(x)) - F(x,u(x))| |f(y,u_n(y))v(y)|}{|x-y|^{\lambda(x,y)}} dxdy \\
&+ \int_{\mathbb{R}^N} \int_{\mathbb{R}^N} \frac{|F(x,u(x))| |f(y,u_n(y))v(y)-f(y,u(y))v(y)|}{|x-y|^{\lambda(x,y)}} dxdy \\
&:=B^{n}_{f} + B^{n}_{F}
\end{align*}
By Proposition \ref{HSL-variable}, 
\begin{align*}
 B^{n}_{f} & \leq C \| F(.,u_n) - F(.,u))\|_{L^{q^{+}}(\mathbb{R}^N)} \| f(.,u_n)v\|_{L^{q^{+}}(\mathbb{R}^N)} \\
 & +C  \| F(.,u_n) - F(.,u))\|_{L^{q^{-}}(\mathbb{R}^N)} \| f(.,u_n)v\|_{L^{q^{-}}(\mathbb{R}^N)}.
\end{align*}
Since the sequences $(\| f(.,u_n)v\|_{L^{q^{+}}(\mathbb{R}^N)})$ and $(\| f(.,u_n)v\|_{L^{q^{-}}(\mathbb{R}^N)})$ are bounded (see \eqref{continuous2} and \eqref{continuous3}) and   
$$
\| F(.,u_n) - F(.,u)\|_{L^{q^{+}}(\mathbb{R}^N)} , \| F(.,u_n) - F(.,u)\|_{L^{q^{-}}(\mathbb{R}^N)} \to 0,
$$
it follows that
\begin{equation}\label{Sup1}
\displaystyle\sup_{\substack{v \in W^{1,p(x)}(\mathbb{R}^N) \\ \| v\|_{W^{1,p(x)}(\mathbb{R}^N) }\leq 1}} \| F(.,u_n) - F(.,u))\|_{L^{q^{+}}(\mathbb{R}^N)} \| f(.,u_n)v\|_{L^{q^{+}}(\mathbb{R}^N)} \to 0
\end{equation}
and
\begin{equation}\label{Sup2}
\displaystyle\sup_{\substack{v \in W^{1,p(x)}(\mathbb{R}^N) \\ \| v\|_{W^{1,p(x)}(\mathbb{R}^N) }\leq 1}} \| F(.,u_n) - F(.,u))\|_{L^{q^{-}}(\mathbb{R}^N)} \| f(.,u_n)v\|_{L^{q^{-}}(\mathbb{R}^N)} \to 0
\end{equation}
as $n \rightarrow + \infty.$

Now we will estimate $B^{n}_{F}.$  Given $\varepsilon >0$,  fix $R>0$ large enough such that 
$$
\int_{B(0,R)^c} |u(x)|^{q^{+}r(x)} dx, \int_{B(0,R)^c} |u(x)|^{q^{-}r(x)} dx, \int_{B(0,R)^c} |u(x)|^{q^{+}s(x)} dx, \int_{B(0,R)^c} |u(x)|^{q^{-}s(x)} dx < \varepsilon .
$$
Recalling that $u_n \to u$ in $L^{q^{+}r(x)}(\mathbb{R}^N),L^{q^{-}r(x)}(\mathbb{R}^N),L^{q^{+}s(x)}(\mathbb{R}^N)$ and $L^{q^{-}s(x)}(\mathbb{R}^N)$, there is $n_0 \in \mathbb{N}$ large enough such that  
\begin{equation}\label{aux-epsilon}
\int_{B(0,R)^c} |u_n(x)|^{q^{+}r(x)} dx, \int_{B(0,R)^c} |u_n(x)|^{q^{-}r(x)} dx, \int_{B(0,R)^c} |u_n(x)|^{q^{+}s(x)} dx, \int_{B(0,R)^c} |u_n(x)|^{q^{-}s(x)} dx < \varepsilon 
\end{equation}
for all $n \geq n_0.$ Note that by Proposition \ref{HSL-variable}
\begin{equation}\label{B-F}
\begin{aligned}
\int_{\mathbb{R}^N}\int_{\mathbb{R}^N} \frac{|F(x,u(x))| |(f(y,u_n(y)) - f(y,u(y)))v(y)|}{|x-y|^{\lambda(x,y)}} dx dy  \leq C\| F(.,u)\|_{L^{q^{+}}(\mathbb{R}^N)} \\
\times\|(f(.,u_n) - f(.,u))v \|_{L^{q^{+}}(\mathbb{R}^N)} \\
+ C\| F(.,u)\|_{L^{q^{-}}(\mathbb{R}^N)} 
\|(f(.,u_n) - f(.,u))v \|_{L^{q^{-}}(\mathbb{R}^N)} \\
\leq C_0(\|(f(.,u_n) - f(.,u))v \|_{L^{q^{+}}(\mathbb{R}^N)} + \|(f(.,u_n) - f(.,u))v \|_{L^{q^{-}}(\mathbb{R}^N)})
\end{aligned}
\end{equation}
where $C_0>0$ is a constant that does not depend on $n \in \mathbb{N}.$ The condition $(f_1)$ together with  H\"{o}lder's inequality yield 
\begin{equation}\label{aux-rc}
\begin{aligned}
&\int_{B(0,R)^c} |(f(y,u_n(y)) - f(x,u(x)))v(y)|^{q^+} dy \leq
 C_1 \| |u_n|^{q^{+}(r(y)-1)}\|_{L^{\frac{r(y)}{r(y)-1}}(B(0,R)^c)} \| |v|^{q^{+}}\|_{L^{r(y)}(B(0,R)^c) } \\
&+ C_1\| |u|^{q^{+}(r(y)-1)}\|_{L^{\frac{r(y)}{r(y)-1}}(B(0,R)^c)} \| |v|^{q^{+}}\|_{L^{r(y)}(B(0,R)^c)} + C_1\| |u_n|^{q^{+}(s(y)-1)}\|_{L^{\frac{s(y)}{s(y)-1}}(B(0,R)^c)}\| |v|^{q^{+}}\|_{L^{s(y)}(B(0,R)^c)}.
\end{aligned}
\end{equation}
By Propositions   \ref{norm-prop} and \ref{cpt-emb}, we know that $\| |v|^{q^{+}}\|_{L^{r(y)}(B(0,R)^c)}, \| |v|^{q^{+}}\|_{L^{s(y)}(B(0,R)^c)} \leq C_2$  where $C_2$ is a positive constant that does not depend on $v \in W^{1,p(x)}(\mathbb{R}^N)$ with $\| v\|_{W^{1,p(x)}(\mathbb{R}^N)} \leq 1$ and $R>0$. Thus, from \eqref{aux-epsilon}, \eqref{aux-rc} and Proposition 2.3 
$$\int_{B(0,R)^c} |(f(y,u_n(y)) - f(x,u(x)))v(y)|^{q^+} dy \leq C_2 \max \left\{\varepsilon^{\frac{1}{\left(\frac{r}{r-1}\right)^+}}, \varepsilon^{\frac{1}{\left(\frac{r}{r-1}\right)^-}}, \varepsilon^{\frac{1}{\left(\frac{s}{s-1}\right)^+}}, \varepsilon^{\frac{1}{\left(\frac{s}{s-1}\right)^-}}\right\}.$$
Therefore
\begin{equation}\label{sup1}
\displaystyle\sup_{\substack{v \in W^{1,p(x)}(\mathbb{R}^N) \\ \| v\|_{W^{1,p(x)}(\mathbb{R}^N) }\leq 1}}  \int_{B(0,R)^c} |(f(y,u_n(y)) - f(x,u(y)))v(y)|^{q^+} dy \leq A_{\varepsilon}, \quad \forall n \geq n_0. 
\end{equation}
where
$$
A_{\varepsilon}=C_2 \max \left\{\varepsilon^{\frac{1}{\left(\frac{r}{r-1}\right)^+}}, \varepsilon^{\frac{1}{\left(\frac{r}{r-1}\right)^-}}, \varepsilon^{\frac{1}{\left(\frac{s}{s-1}\right)^+}}, \varepsilon^{\frac{1}{\left(\frac{s}{s-1}\right)^-}}\right\} \to 0 \quad \mbox{as} \quad \varepsilon \to 0.
$$

Now we will estimate the integral 
$$
\int_{B(0,R)} |f(y,u_n (y)) - f(y,u(y))|^{q^+}|v(y)|^{q^+} dy .
$$
Using Proposition \ref{HSL-variable} and the continuous embedding $W^{1,p(x)}(B(0,R)) \hookrightarrow L^{q^{+}r(y)}(B(0,R))$, we get
\begin{equation*}
\begin{aligned}
\int_{B(0,R)} |f(y,u_n(y)) - f(y,u(y))|^{q^{+}} |v(y)|^{q^{+}} & \leq C \| |f(.,u_n) - f(.,u)|^{q^{+}}\|_{L^{\frac{r(y)}{r(y)-1}}(B(0,R) )}\\
& \times \| |v|^{q^{+}}\|_{L^{r(y)}(B(0,R))}\\
& \leq C_3 \||f(.,u_n) - f(.,u)|^{q^{+}} \|_{L^{\frac{r(y)}{r(y)-1}}(B(0,R))}
\end{aligned}
\end{equation*}
where $C_3$ is a positive constant that does not depend on $n \in \mathbb{N}$ and $v \in W^{1,p(x)}(\mathbb{R}^N)$ with $\| v\|_{W^{1,p(x)}(\mathbb{R}^N)} \leq 1.$ Recalling that $|f(y,u_n(y) - f(y,u(y)))|^{q^{+}} \rightarrow 0$ in $L^{\frac{r(y)}{r(y)-1}}(B(0,R))$, it follows that 
\begin{equation}\label{sup2}
\displaystyle\sup_{\substack{v \in W^{1,p(x)}(\mathbb{R}^N) \\ \| v\|_{W^{1,p(x)}(\mathbb{R}^N) }\leq 1}}  \int_{B(0,R)} |(f(y,u_n(y)) - f(x,u(y)))v(y)|^{q^+} dy \rightarrow 0 \ \text{as} \ n \rightarrow + \infty.
\end{equation}
From $\eqref{sup1}$ and $\eqref{sup2}$, 
\begin{equation}\label{sup3}
\displaystyle\sup_{\substack{v \in W^{1,p(x)}(\mathbb{R}^N) \\ \| v\|_{W^{1,p(x)}(\mathbb{R}^N) }\leq 1}} \|(f(.,u_n) - f(.,u))v \|_{L^{q^{+}}(\mathbb{R}^N)} \rightarrow 0 \  \text{as} \ n \rightarrow + \infty .
\end{equation}
Likewise
\begin{equation}\label{sup4}
\displaystyle\sup_{\substack{v \in W^{1,p(x)}(\mathbb{R}^N) \\ \| v\|_{W^{1,p(x)}(\mathbb{R}^N) }\leq 1}} \|(f(.,u_n) - f(.,u))v \|_{L^{q^{-}}(\mathbb{R}^N)} \rightarrow 0 \  \text{as} \ n \rightarrow + \infty .
\end{equation}
From \eqref{B-F},\eqref{sup3} and \eqref{sup4},  

\begin{equation}\label{sup5}
\displaystyle\sup_{\substack{v \in W^{1,p(x)}(\mathbb{R}^N) \\ \| v\|_{W^{1,p(x)}(\mathbb{R}^N) }\leq 1}}\int_{\mathbb{R}^N}\int_{\mathbb{R}^N} \frac{|F(x,u(x))| |(f(y,u_n(y)) - f(y,u(y)))v(y)|}{|x-y|^{\lambda(x,y)}} dx dy \to 0
\end{equation}
as $n \rightarrow +\infty.$ The step is justified, according to \eqref{Sup1}, \eqref{Sup2} and \eqref{sup5}. Finally, the lemma follows from the previous three steps.

\end{proof}

\section{An application} 

In this section we will illustrate how we can use  Proposition \ref{HSL-variable} to prove the existence of solution for (\ref{Choquard-eq}). In what follows, we will consider  the condition $(f_1)$ with $r,s$ verifying
\begin{equation} \label{NOVAEQ}
p \ll rq^{-} \leq rq^{+} \ll p^{\star}, \,\, p \ll sq^{-} \leq sq^{+} \ll p^{\star} 
\end{equation}
and
\begin{equation} \label{growth3.1}
r^{-},s^{-} > {p^{+}}/{2}.
\end{equation}
Moreover, we also consider the  Ambrosetti-Rabinowitz type condition:
$$
0<\theta F(x,t)  \leq 2f(x,t)t, \quad \forall  t>0  \eqno{(f_2)}
$$
Here $\theta>0$ is a fixed number with $\theta > p^{+}$  and we will suppose that there are constants $l,c_l>0$ such that  
$$
F(x,l) \geq c_l,  \quad  \forall x \in \mathbb{R}^N.
$$

Related to the potential $V: \mathbb{R}^N \rightarrow \mathbb{R}$, we assume that    
$$
\displaystyle\inf_{x \in \mathbb{R}^N} V(x):=V_0 >0  \eqno{(V_0)}
$$
and one of the following conditions: 

\noindent $(V_1)$ $V$ is $\mathbb{Z}^N$-periodic \\
\noindent or \\
\noindent $(V_2)$ $V$ has the property that the quantity
$$
\| u\|_{\star} = \| \nabla u\|_{L^{p(x)}(\mathbb{R}^N)} + \| u\|_{L^{p(x)}(\mathbb{R}^N),V(x)}
$$
where
$$
\| u\|_{L^{p(x)}(\mathbb{R}^N),V(x)}= \inf\left\{\alpha >0 ; \int_{\mathbb{R}^N} V(x) \left| \frac{u}{\alpha}\right| \leq 1 \right\}, 
$$
defines a norm in $C^{\infty}_{0}(\mathbb{R}^N)$ and that the completion of $C^{\infty}_{0}(\mathbb{R}^N)$ with relation this norm, denote by $E$, is a Banach space with the embedding $E \hookrightarrow L^{q(x)}(\mathbb{R}^N)$ compact for all $q \in C^{+}(\mathbb{R}^N)$ and $p \ll q \ll p^{\star}$ in $\mathbb{R}^N .$

Note that if we consider the conditions $(V_1)$ and $(V_2)$, the same arguments of Lemma \eqref{C1} work well to prove that  $\Psi$ given by \eqref{nonlocal-func} belongs to $C^1(E,\mathbb{R}).$

We would like point out that the condition $(V_2)$ holds if the potential $V$ is coercive, that is 
$$
V(x) \to +\infty \quad \mbox{as} \quad |x| \to +\infty,
$$	
see for instance \cite{Alves3}.

The main result of this section is the following:

\begin{theorem} \label{existence} Assume $(f_1)-(f_2)$, $(V_0),$ (\ref{symetry}),  (\ref{NOVAEQ})- (\ref{growth3.1}) and that $p$ is a Lipschitz functions. If  \\
\noindent $i)$ \,  $(V_1)$ holds,  $p,q$ are $\mathbb{Z}^N$-periodic functions and $f(.,t)$ is a  $\mathbb{Z}^N$-periodic function for each $t \in \mathbb{R}$,\\
\noindent or \\
\noindent $ii)$  $(V_2)$ holds, \\
\noindent then problem (\ref{Choquard-eq}) has a nontrivial solution. 
\end{theorem}


In the proof of Theorem \ref{existence} we will use variational methods. From now on $(A, \| \,\,\, \|)$ denotes $(W^{1,p(x)}(\mathbb{R}^N),\|\,\,\,\|_{W^{1,p(x)}(\mathbb{R}^N)})$ or $(E,\| \,\,\, \|_{\star})$. The energy functional $J:A \to \mathbb{R}$ associated with  (\ref{Choquard-eq}) is given by,
$$
J(u)=\int_{\mathbb{R}^N}\frac{1}{p(x)}(|\nabla u(x)|^{p(x)}+V(x)|u(x)|^{p(x)})\,dx-\frac{1}{2}\int_{\mathbb{R}^N} \int_{\mathbb{R}^N} \frac{F(x,u(x))F(y,u(y))}{|x-y|^{\lambda(x,y)}}\,dx dy,
$$ 
that is,
$$
J(u)=\int_{\mathbb{R}^N}\frac{1}{p(x)}(|\nabla u(x)|^{p(x)}+V(x)|u(x)|^{p(x)})\,dx-\Psi(u). 
$$
By the study made in the previous section, $J \in C^{1}(A,\mathbb{R})$ with
\begin{equation*}
\begin{aligned}
J'(u)v=\int_{\mathbb{R}^N}|\nabla u(x)|^{p(x)-2}\nabla u(x) \nabla v(x) \, dx &+ \int_{\mathbb{R}^N} V(x)|u(x)|^{p(x)-2}u(x)v(x) dx \\
&- \int_{\mathbb{R}^N} \int_{\mathbb{R}^N} \frac{F(x,u(x)f(y,u(y)) v(y))}{|x-y|^{\lambda(x,y)}} dxdy, \quad \forall u,v \in A.
\end{aligned}
\end{equation*}

Our first lemma establishes the mountain pass geometry. 

\begin{lemma}\label{first-geo} The functional $J$ verifies the following properties:
\begin{enumerate}[(i)]
	
\item 	There exists $\rho>0$ small enough such that $J(u) \geq \eta$ for $u \in A$ with $\| u\|=\rho$ for some $\eta>0.$
\item There exists $e \in A$ such that $\| e\| > \rho$ and $J(e) <0.$

\end{enumerate}
\end{lemma}
\begin{proof}
$i)$	
By Proposition \ref{HSL-variable} and  $\eqref{symetry}$,  
$$
\int_{\mathbb{R}^N} \int_{\mathbb{R}^N} \frac{F(x,u(x))F(y,u(y))}{|x-y|^{\lambda(x,y)}} dx dy\leq C (\| F(.,u)\|^{2}_{L^{q^{+}}(\mathbb{R}^N)} + \| F(.,u)\|^{2}_{L^{q^{-}}(\mathbb{R}^N)})
$$
for all $u \in A$. Note that 
\begin{align*}
\| F(.,u)\|_{L^{q^{+}}(\mathbb{R}^N)} & \leq C \left( \int_{\mathbb{R}^N} |u(x)|^{q^{+}r(x)} + |u(x)|^{q^{+}s(x)} dx \right)^{\frac{1}{q^{+}}} \\
&\leq C\left(\int_{\mathbb{R}^N} |u(x)|^{q^{+}r(x)} dx \right)^{\frac{1}{q^{+}}} + C\left(\int_{\mathbb{R}^N} |u(x)|^{q^{+}s(x)} dx \right)^{\frac{1}{q^{+}}}
\end{align*}
$$
\leq C\left(\max \left( \| u\|^{r^{+}}_{L^{q^{+}r(x)}(\mathbb{R}^N)} , \| u\|^{r^{-}}_{L^{q^{+}r(x)}(\mathbb{R}^N)}\right) + \max \left( \| u\|^{s^{+}}_{L^{q^{+}s(x)}(\mathbb{R}^N)} , \| u\|^{s^{-}}_{L^{q^{+}s(x)}(\mathbb{R}^N)}\right)\right)
$$
and
\begin{align*}
\| F(.,u)\|_{L^{q^{-}}(\mathbb{R}^N)} &\leq C\max \left( \| u\|^{r^{+}}_{L^{q^{-}r(x)}(\mathbb{R}^N)} , \| u\|^{r^{-}}_{L^{q^{-}r(x)}(\mathbb{R}^N)}\right)\\
&+ C\max \left( \| u\|^{s^{+}}_{L^{q^{+}s(x)}(\mathbb{R}^N)} , \| u\|^{s^{-}}_{L^{s^{-}r(x)}(\mathbb{R}^N)}\right),  
\end{align*}
according to $(F)$. The continuous embeddings $A \hookrightarrow W^{1,p(x)}(\mathbb{R}^N)$ and $W^{1,p(x)}(\mathbb{R}^N) \hookrightarrow L^{p^{\star}(x)}(\mathbb{R}^N)$ implies that 
$$
\| u\|_{L^{q^{+}r(x)}(\mathbb{R}^N)}, \| u\|_{L^{q^{-}r(x)}(\mathbb{R}^N)}, \| u\|_{L^{q^{+}s(x)}(\mathbb{R}^N)}, \| u\|_{L^{q^{-}s(x)}(\mathbb{R}^N)} \leq L \| u\|_{W^{1,p(x)}(\mathbb{R}^N)}, u \in A
$$
for a positive constant $L>0$ that does not depend on $ u \in A.$

By using the classical inequality 
$$
(a+b)^{\alpha} \leq 2^{\alpha-1}(a^{\alpha} + b^{\alpha}), \,\, a,b>0 \quad \mbox{with} \quad \alpha >1,
$$ 
we get
\begin{align*}
J(u) & \geq \int_{\mathbb{R}^N} \frac{1}{p^{+}} \left( |\nabla u(x)|^{p(x)} + V_0 |u(x)|^{p(x)}\right) dx - C\max(\| u\|_{W^{1,p(x)}(\mathbb{R}^N)}^{2r^{-}}, \| u\|_{W^{1,p(x)}(\mathbb{R}^N)}^{2r^{+}} )\\
& - C\max(\| u\|_{W^{1,p(x)}(\mathbb{R}^N)}^{2s^{-}}, \| u\|_{W^{1,p(x)}(\mathbb{R}^N)}^{2s^{+}} )\\
& \geq \overline{C}(\| \nabla u\|^{p^{+}}_{L^{p(x)}(\mathbb{R}^N)} + \|  u\|^{p^{+}}_{L^{p(x)}(\mathbb{R}^N)} ) - C(\| u\|_{W^{1,p(x)}(\mathbb{R}^N)}^{2r^{+}} + \| u\|_{W^{1,p(x)}(\mathbb{R}^N)}^{2r^{-}}) \\
&- C(\| u\|_{W^{1,p(x)}(\mathbb{R}^N)}^{2s^{+}} + \| u\|_{W^{1,p(x)}(\mathbb{R}^N)}^{2s^{-}}) \\
& \geq \overline{C}(\| \nabla u\|^{p^{+}}_{L^{p(x)}(\mathbb{R}^N)} + \|  u\|^{p^{+}}_{L^{p(x)}(\mathbb{R}^N)} ) - K(\| \nabla u\|_{L^{p(x)}(\mathbb{R}^N)}^{2r^{+}} + \| u\|_{L^{p(x)}(\mathbb{R}^N)}^{2r^{+}}) \\
& -K(\| \nabla u\|_{L^{p(x)}(\mathbb{R}^N)}^{2r^{-}} + \| u\|_{L^{p(x)}(\mathbb{R}^N)}^{2r^{-}})  -K(\| \nabla u\|_{L^{p(x)}(\mathbb{R}^N)}^{2s^{+}} + \| u\|_{L^{p(x)}(\mathbb{R}^N)}^{2s^{+}}) \\
& -K(\| \nabla u\|_{L^{p(x)}(\mathbb{R}^N)}^{2s^{-}} + \| u\|_{L^{p(x)}(\mathbb{R}^N)}^{2s^{-}})
\end{align*}
where $C, \overline{C}$  and $K$ are constants that does not depend on $u.$
Since $2r^{-},2s^{-} > p^{+}$ and $\| u\|_{W^{1,p(x)}(\mathbb{R}^N)}   = \| \nabla u\|_{L^{p(x)}(\mathbb{R}^N)} + \| u\|_{L^{p(x)}(\mathbb{R}^N)} \leq \overline{K}\| u\|$ where $\overline{K} $ is a constant that does not depend on $u.$ The result follows by fixing $\| u\|=\rho$ with $\rho$ small enough.
\\

ii)
The condition $(f_2)$ implies that 
$$
F(x,t) \geq C t^{\frac{\theta}{2}} \quad \forall (x,t) \in \mathbb{R}^N \times \mathbb{R} \quad \mbox{and} \quad t \geq l,
$$ 
where $C$ depends only on $l$ and $\theta.$ Now, considering a nonnegative function $\varphi \in C^{\infty}_{0}(\mathbb{R}^N) \setminus \{0\}$ the last inequality permits to conclude that $J(t \varphi)<0$ for $t$ large enough. This finishes the proof.

\end{proof}
Using the Mountain Pass Theorem without the Palais-Smale condition (see \cite[Theorem 5.4.1]{Chab}), there is a sequence $(u_n)\subset A$ such that 
$$
J(u_n ) \rightarrow d \quad \mbox{and} \quad J^{'}(u_n) \rightarrow 0,
$$ 
where $d>0$ is the mountain pass level defined by
\begin{equation}\label{level}
 d:= \inf_{\gamma \in \Gamma} \sup_{t \in [0,1]} J(\gamma (t))
\end{equation}
with $\Gamma:=\{ \gamma \in C([0,1], A); \gamma(0)=0, \gamma (1)=e\}$.

\vspace{0.5 cm}

Regarding such sequence we have the next result.

\begin{lemma}
The sequence $(u_n)$ is bounded.
\end{lemma}
\begin{proof}
First of all, note that 
$$
J(u_n)  - \frac{J^{'}(u_n)u_n}{\theta} \leq d +1+  \| u_n\|.
$$
for $n$	large enough and $d$ given in \eqref{level}. On the other hand,  
\begin{align*}
J(u_n) - \frac{J^{'}(u_n)u_n}{\theta} & = \int_{\mathbb{R}^N} \left( \frac{1}{p(x)} - \frac{1}{\theta}\right)(| \nabla u_n (x)|^{p(x)} + V(x) |u_n (x)|^{p(x)}) dx \\
&+ \int_{\mathbb{R}^N} \int_{\mathbb{R}^N} \frac{F(x,u_n(x))}{|x-y|^{\lambda(x,y)}} \left( \frac{f(y,u_n(y))u_n(y)}{\theta}  - \frac{F(y,u_n(y))}{2}\right) dx dy \\
& \geq C \int_{\mathbb{R}^N} |\nabla u_n (x)|^{p(x)} + V(x)|u_n (x)|^{p(x)} dx. 
\end{align*}
The last two inequalities give the boundedness of $(u_n)$ in $A.$
\end{proof}

Since $(u_n)$ is bounded in $A$,  $(u_n)$ is also bounded in $W^{1,p(x)}(\mathbb{R}^N)$. The compact embeddings contained in Proposition \ref{cpt-emb} imply that exists $u \in A$ and a subsequence, still denoted by $(u_n)$, such that $u_n(x) \rightarrow u(x) $ a.e in $\mathbb{R}^N$ and $\nabla u_n \rightharpoonup \nabla u$ in $(L^{p(x)}(\mathbb{R}^N))^N ,$ where the symbol $\rightharpoonup $ denotes the weak convergence. 

\vspace{0.5 cm}

The next two lemmas will be needed to prove that $u$ is a critical point  of $J$. 
\begin{lemma} \label{CONVEGENCIA}
The following limits hold for some subsequence:	
\begin{enumerate}[(i)]
\item	
$$
\int_{\mathbb{R}^N} \int_{\mathbb{R}^N} \frac{F(x,u_n(x)) f(y,u(y))v(y)}{|x-y|^{\lambda(x,y)}}dx dy \to \int_{\mathbb{R}^N} \int_{\mathbb{R}^N} \frac{F(x,u(x)) f(y,u(y))v(y)}{|x-y|^{\lambda(x,y)}}dx dy 
$$
for all $v \in C^{\infty}_{0}(\mathbb{R}^N)$
\item
$$
 \int_{\mathbb{R}^N} \int_{\mathbb{R}^N} \frac{F(x,u_n(x))(f(y,u_n(y)) v(y) - f(y,u(y))v(y))}{|x-y|^{\lambda(x,y)}} dx dy \to 0
$$
for all $v \in C^{\infty}_{0}(\mathbb{R}^N).$
\item $$ \int_{\mathbb{R}^N} \int_{\mathbb{R}^N} \frac{F(x,u_n(x)) f(y,u_n(y))v(y)}{|x-y|^{\lambda(x,y)}} dx dy \rightarrow \int_{\mathbb{R}^N} \int_{\mathbb{R}^N} \frac{F(x,u(x)) f(y,u(y))v(y)}{|x-y|^{\lambda(x,y)}} dx dy$$
for all $v \in C^{\infty}_{0}(\mathbb{R}^N).$
\end{enumerate}
\end{lemma}
\begin{proof}
	
\textit{i)} As $L^{q^{+}}(\mathbb{R}^N)$ and  $L^{q^{-}}(\mathbb{R}^N)$ are uniformly convex, the Banach space $(L^{q^{+}}(\mathbb{R}^N)\cap L^{q^{+}}(\mathbb{R}^N), \max (\| . \|_{L^{q^{+}}(\mathbb{R}^N)} , \| . \|_{L^{q^{-}}(\mathbb{R}^N)} )   )$ is uniformly convex (therefore reflexive). The  growth of $F$  and the fact that $(u_n)$ is bounded in  $A$  ensures that the sequence $(F(.,u_n(.)))$ is bounded in $L^{q^{+}}(\mathbb{R}^N) \cap L^{q^{-}}(\mathbb{R}^N)$.

We claim that $F(.,u_n) \rightharpoonup F(.,u)$ in $L^{q^{+}}(\mathbb{R}^N) \cap L^{q^{-}}(\mathbb{R}^N).$ Since $(F(.,u_n))$ is bounded in $L^{q^{+}}(\mathbb{R}^N) \cap L^{q^{-}}(\mathbb{R}^N)$ there exists $L \in L^{q^{+}}(\mathbb{R}^N) \cap L^{q^{-}}(\mathbb{R}^N)$ such that $F(.,u_n) \rightharpoonup L$ in $L^{q^{+}}(\mathbb{R}^N) \cap L^{q^{-}}(\mathbb{R}^N)$. Fix $\varphi \in C^{\infty}_{0}(\mathbb{R}^N)$ and consider the continuous linear functional 
$$
I_{\varphi}(w):= \int_{\mathbb{R}^N} w\varphi \ dx,w \in L^{p^{+}}(\mathbb{R}^N) \cap L^{p^{-}}(\mathbb{R}^N). 
$$ 
Then, 
$$
I_{\varphi}(F(.,u_n)) \rightarrow \int_{\mathbb{R}^N} L(x)\varphi(x) dx.
$$ 
Using \cite[Proposition 2.6]{AlvesFerreira}, $F(.,u_n) \rightharpoonup F(.,u)$ in $L^{q^{+}}(\mathbb{R}^N)$, and so,  
$$
\int_{\mathbb{R}^N}  F(x,u_n(x)) \varphi dx \rightarrow \int_{\mathbb{R}^N}  F(x,u(x)) \varphi dx. 
$$
Thereby
$$
\int_{\mathbb{R}^N}  F(x,u(x)) \varphi (x)dx =  \int_{\mathbb{R}^N} L(x)\varphi (x) dx
\quad \forall \varphi \in C^{\infty}_{0}(\mathbb{R}^N),  
$$
showing that
$$
L(x) = F(x,u(x)) \quad \mbox{ a.e in } \quad \mathbb{R}^N.
$$  
By Proposition \ref{HSL-variable}, the application
$$H(w):=\int_{\mathbb{R}^N} \int_{\mathbb{R}^N}  \frac{w(x)f(y,u(y))v(y)}{|x-y|^{\lambda(x,y)}} , w \in L^{q^{+}}(\mathbb{R}^N) \cap L^{q^{-}}(\mathbb{R}^N) $$
is a continuous linear functional. Since $F(.,u_n) \rightharpoonup F(.,u)$ in $L^{q^{+}}(\mathbb{R}^N) \cap L^{q^{-}}(\mathbb{R}^N)$, it follows that  
$$ 
\int_{\mathbb{R}^N} \int_{\mathbb{R}^N} \frac{F(x,u_n(x)) f(y,u_n(y))v(y)}{|x-y|^{\lambda(x,y)}}dx dy \rightarrow \int_{\mathbb{R}^N} \int_{\mathbb{R}^N} \frac{F(x,u(x)) f(y,u(y))v(y)}{|x-y|^{\lambda(x,y)}}dx dy,
$$
which proves $i)$. 

\textit{ii)} Denote by $I$ the integral described in $ii).$
Then
\begin{align*}
|I| & \leq C \| F(.,u_n)\|_{L^{q^{+}}(\mathbb{R}^N)} \| f(.,u_n)v - f(.,u)v\|_{L^{q^{+}}(\mathbb{R}^N)} \\
&+ \| F(.,u_n)\|_{L^{q^{-}}(\mathbb{R}^N)} \| f(.,u_n)v - f(.,u)v\|_{L^{q^{-}}(\mathbb{R}^N)},
\end{align*}
according  to Proposition \ref{HSL-variable} with $p(\cdot) = q(\cdot)$. Since $(u_n)$ is bounded in $A$, $(F(.,u_n))$ is also bounded sequence in $L^{q^{+}}(\mathbb{R}^N) \cap  L^{q^{-}}(\mathbb{R}^N).$ Let $v \in C^{\infty}_{0}(\mathbb{R}^N)$ and consider a bounded open set $\Omega$ that contains the support of $v.$  Since $\Omega$ is bounded, the boundedness of $(u_n)$ in $A$ combined with (\ref{NOVAEQ}) and the compact embeddings given in Proposition \ref{cpt-emb} guarantee that for some a subsequence, 
\begin{itemize}
	\item $u_n \to u$ in $ L^{q^{+}r(x)}(\Omega),$
	\item $u_n \to u$ in $ L^{q^{+}s(x)}(\Omega),$
	\item $u_n(x) \rightarrow u(x) $ a.e in $\Omega,$
	\item $|u_n(x)| \leq h_1(x)$ a.e in $\Omega$ for some $h_1 \in L^{q^{+}r(x)}(\Omega),$
	\item $|u_n(x)| \leq h_2(x)$ a.e in $\Omega$ for some $h_2 \in L^{q^{+}s(x)}(\Omega).$
\end{itemize}
These informations combined with Lebesgue's Dominated Convergence Theorem give
$$
 \| f(.,u_n)v - f(.,u)v\|_{L^{q^{+}}(\mathbb{R}^N)}=\| f(.,u_n)v - f(.,u)v\|_{L^{q^{+}}(\Omega)} \rightarrow 0.
$$
A similar reasoning provides
$$
\| f(.,u_n)v - f(.,u)v\|_{L^{q^{-}}(\mathbb{R}^N)}=\| f(.,u_n)v - f(.,u)v\|_{L^{q^{-}}(\Omega)} \rightarrow 0.
$$
This finishes the proof of $ii)$. 
\textit{iii)} is a direct consequence of \textit{i)} and \textit{ii)}.
\end{proof}

\begin{lemma}\label{lemma-3}
For a subsequence the two properties below hold
\begin{enumerate}[(i)]
\item $\nabla u_n(x) \rightarrow \nabla u(x)$  a.e in $\mathbb{R}^N.$
\item $|\nabla u_n|^{p(x)-2} \nabla u_n \rightharpoonup |\nabla u_n|^{p(x)-2} \nabla u_n$ in $(L^{\frac{p(x)}{p(x)-1}}({\mathbb{R}}^n))^N.$
\end{enumerate}
\end{lemma}
\begin{proof}
Fix $R>0$ and $\varphi \in C^{\infty}_{0}(\mathbb{R}^N)$ such that $\varphi(x)=1$ for $x \in B_R(0).$ Since $J^{'}(u_n) \rightarrow 0$ in $A{'}$ and $(u_n)$ is bounded in $A$, we have $J^{'}(u_n)(u_n \varphi) = o_n(1)=J^{'}(u_n)(u \varphi).$
Setting 
$$
P_n(x):= ( |\nabla u_n(x)|^{p(x)-2}\nabla u_n (x) - |\nabla u(x)|^{p(x)-2} \nabla u(x)). (\nabla u_n(x)-\nabla u(x)), x \in \mathbb{R}^N,
$$
we derive 
\begin{equation*}
\begin{aligned}
&\int_{\mathbb{R}^N} P_n (x) \varphi (x) dx = J^{\prime}(u_n).(u_n \varphi) - \int_{\mathbb{R}^N} u_n(x) |\nabla u_n(x)|^{p(x)-2} \nabla u_n(x) \nabla \varphi(x) dx \\
&-\int_{\mathbb{R}^N} |u_n(x)|^{p(x)}\varphi(x) dx -J^{\prime} (u_n)(u\varphi) + \int_{\mathbb{R}^N} u(x)|\nabla u_n(x)|^{p(x)-2}\nabla u_n(x) \nabla \varphi(x) dx \\
&+ \int_{\mathbb{R}^N} |u_n(x)|^{p(x)-2} u_n(x) u(x) \varphi dx - \int_{\mathbb{R}^N} \int_{\mathbb{R}^N} \frac{F(x,u_n(x)) f(y,u_n(y)) (u_n(y)-u(y))v(y)}{|x-y|^{\lambda(x,y)}} dx dy \\
&- \int_{\mathbb{R}^N} |\nabla u(x)|^{p(x)-2}\nabla u(x) \nabla (u_n -u)(x) \varphi dx.
\end{aligned}
\end{equation*}
Standard arguments ensure that 
$$
\int_{\mathbb{R}^N} P_n(x) \varphi(x)\,dx \rightarrow 0.
$$
Therefore $\nabla u_n \rightarrow \nabla u$ in $L^{p(x)}(B_R(0))$ for all $R>0$. As $R$ is arbitrary,  we conclude that $\nabla u_n(x) \rightarrow \nabla u(x)$ a.e in $\mathbb{R}^N$ for some subsequence.

\textit{ii)} Using the fact that $|\nabla u_n|^{p(x)-2} \nabla u_n$ is bounded in $L^{\frac{p(x)}{p(x)-1}}(\mathbb{R}^N)$ and the pointwise convergence  $|\nabla u_n(x)|^{p(x)-2} \nabla u_n(x) \rightarrow |\nabla u(x)|^{p(x)-2} \nabla u(x)$ a.e in $\mathbb{R}^N$, we have  
$$
|\nabla u_n |^{p(x)-2} \nabla u_n \rightharpoonup |\nabla u_n |^{p(x)-2} \nabla u_n \ \text{in} \ (L^{\frac{p(x)}{p(x)-1}}({\mathbb{R}}^N))^N.
$$
according to \cite[Proposition 2.6]{AlvesFerreira}. 
\end{proof}

Now, we are ready to prove that $u$ is a critical point of $J.$

\begin{lemma}\label{critical-point} The function $u$ is a critical point of $J$, that is, $J^{'}(u)=0$.
\end{lemma}
\begin{proof}
First of all, we claim that 
$$
J^{'}(u_n)v \to  J^{'}(u)v, \quad \forall v \in C^{\infty}_{0}(\mathbb{R}^N).
$$
In order to verify such limit, note that
\begin{align*}
J^{'}(u_n)v&= \int_{\mathbb{R}^N} |\nabla u_n(x)|^{p(x)-2}  \nabla u_n(x)\nabla v(x) + V(x)|u_n|^{p(x)-2}u_n(x) v(x) dx \\
&- \int_{\mathbb{R}^N} \int_{\mathbb{R}^N}  \frac{F(x,u_n(x))f(y,u_n(y))v(y)}{|x-y|^{\lambda(x,y)} } dx dy.
\end{align*}
By Lemmas \ref{CONVEGENCIA}  and \ref{lemma-3} we get

\begin{equation}\label{lim-HSL}
\int_{\mathbb{R}^N} \int_{\mathbb{R}^N} \frac{F(x,u_n(x)) f(y,u_n(y))v(y)}{|x-y|^{\lambda(x,y)}} dx dy \rightarrow \int_{\mathbb{R}^n} \int_{\mathbb{R}^N} \frac{F(x,u(x)) f(y,u(y))v(y)}{|x-y|^{\lambda(x,y)}} dx dy.
\end{equation}
and
\begin{equation}\label{lim-marc}
\int_{\mathbb{R}^N} |\nabla u_n (x)|^{p(x)-2} \nabla u_n(x)\nabla v(x)  dx \rightarrow \int_{\mathbb{R}^N} |\nabla u(x)|^{p(x)-2} \nabla u(x).\nabla v(x)  dx.
\end{equation}
Moreover, by Lebesgue's Dominated Convergence Theorem  we also have 
$$ 
\int_{\mathbb{R}^N} V(x) |u_n(x)|^{p(x)-2}u_n(x) v(x) dx \rightarrow \int_{\mathbb{R}^N} V(x) |u(x)|^{p(x)-2}u(x) v(x) dx,
$$
which combined with the relations  \eqref{lim-HSL} and \eqref{lim-marc} proves the claim. As $J'(u_n)v \to 0$, the claim ensures that $J^{'}(u)v = 0$ for all $v \in C^{\infty}_{0}(\mathbb{R}^N)$. Now, the lemma follows by using the fact that  $C^{\infty}_{0}(\mathbb{R}^N)$ is dense in $A$. 
\end{proof}

\subsection{Proof of Theorem \ref{existence}}. 
In the sequel, we will divide the proof into two cases, which are related to the conditions $(V_1)$ and $(V_2)$.\\

\noindent \textbf{Case 1:} \, $(V_1)$ holds: \\

If $u\neq 0 $, then $u$ is a nontrivial solution and the theorem is proved. If $u=0$, we must find another solution $v \in W^{1,p(x)}(\mathbb{R}^N) \setminus\{ 0\} $ for the equation \eqref{Choquard-eq}. For such purpose, the claim  below is crucial in our argument. 

\begin{claim}\label{lions-conseq}
There exist $r>0,\beta>0$ and a sequence $(y_n) \subset \mathbb{R}^N$ such that
$$
\liminf_{n \rightarrow + \infty} \int_{B_r(y_n)} |u_n(x)|^{p(x)} dx \geq \beta>0.
$$
\end{claim}
\begin{proof}
Suppose that the claim is false. Then, by Lemma \ref{lions-variable}  
 \begin{equation}\label{lions-app}
 u_n \rightarrow 0 \quad \mbox{in} \quad  L^{t(x)}(\mathbb{R}^n),
 \end{equation}
for all $t \in C^{+}(\mathbb{R}^N)$ with $p \ll t \ll p^{\star}$. Applying Proposition \ref{HSL-variable}, 
\begin{align*}
\left| \int_{\mathbb{R}^N} \int_{\mathbb{R}^N} \frac{F(x,u_n(x))f(y,u_n(y))u_n(y)}{|x-y|^{\lambda(x,y)}} dx dy\right |&\leq C \| F(x,u_n(x))\|_{L^{q^{+}}(\mathbb{R}^N)} \| f(y,u_n(y))u_n(y)\|_{L^{q^{+}}(\mathbb{R}^N)} \\
&+  C\| F(x,u_n(x))\|_{L^{q^{-}}(\mathbb{R}^N)} \| f(y,u_n(y))u_n(y)\|_{L^{q^{-}}(\mathbb{R}^N)} .
\end{align*}
By $(f_1)$, $(F)$, (\ref{NOVAEQ}) and \eqref{lions-app}, 
$$
\int_{\mathbb{R}^N} |F(x,u_n(x))|^{q^{+}} dx  \rightarrow 0, 
$$
$$
\int_{\mathbb{R}^N} |F(x,u_n(x))|^{q^{-}} dx \rightarrow 0, 
$$
$$ 
\int_{\mathbb{R}^N} |f(y,u_n(y))u_n(y)|^{q^{+}} dy \rightarrow 0
$$
and
$$ \int_{\mathbb{R}^N} |f(y,u_n(y))u_n(y)|^{q^{-}} dy \rightarrow 0.
$$
Therefore
$$
\int_{\mathbb{R}^N} \frac{F(x,u_n(x))f(y,u_n(y))u_n(y)}{|x-y|^{\lambda(x,y)}} dx dy \rightarrow 0.
$$
The above limit together with the fact that $J'(u_n)u_n=o_n(1)$ give 
$$ 
\int_{\mathbb{R}^N} (|\nabla u_n(x)|^{p(x)} + V(x) |u_n(x)|^{p(x)}) \, dx \to 0,
$$
or equivalently,
$$
u_n \to 0 \quad \mbox{in} \quad A.
$$
This limit leads to $J(u_n) \rightarrow 0$, which contradicts the limit $J(u_n) \rightarrow d >0$.
\end{proof}

By using standard arguments, we can assume  in Claim \ref{lions-conseq} that  $(y_n) \subset \mathbb{Z}^N$. As $q$ is $\mathbb{Z}^N$-periodic, the Remark \ref{RMK1} yields $\lambda$ is $\mathbb{Z}^N \times \mathbb{Z}^N$-periodic. This fact combined with the periodicity of $p, V, f(.,t)$ and $F(.,t)$ guarantee that the function $v_n(x)=u_n(x+y_n)$ satisfies 
$$
J(v_n) = J(u_n), \|J'(v_n)\|=\|J'(u_n)\|  \quad \mbox{and} \quad \| u_n\| = \| v_n\| \quad \forall n \in \mathbb{N}.
$$ 
From the above information, $(v_n)$ is a $(PS)_d$ sequence for $J$. Since $(v_n)$ is bounded in $A$,  up to a subsequence, $v_n \rightarrow v$ in $L^{p(x)}(B_r(0))$ for some $v \in A.$ In order to verify that $v \neq 0$, note that by Claim \ref{lions-conseq}
$$
0<\beta \leq \lim_{n \rightarrow +\infty} \int_{B_r(y_n)} |u_n(x)|^{p(x)} dx 
= \lim_{n \rightarrow +\infty} \int_{B_r(0)} |v_n(x)|^{p(x)} dx  
= \int_{B_r(0)} |v(x)|^{p(x)} dx.
$$

\noindent \textbf{Case 2:} $(V_2)$ holds:\\
 
To begin with, recall that there is a sequence $(u_n) \subset E$ such that
$$
J(u_n) \to d \quad \mbox{and} \quad J'(u_n) \to 0,
$$  
where $d>0$ is the mountain pass level given by \eqref{level}. Since $(u_n)$ is a bounded sequence in $E$, we can assume that for some subsequence of $(u_n)$, still denote by itself,  there is $u \in E$ such that $u_n \rightharpoonup u$ in $E.$
By $(V_2)$,
$$
u_n \to u \quad \mbox{in} \quad L^{s(x)}(\mathbb{R}^N) \quad \ \text{for all} \ s \in C^{+}(\mathbb{R}^N) \ \text{with}  \ p \ll s \ll p^{\star}.
$$
Suppose that $u \equiv 0$. The last limit combined with $(f_1)-(f_2)$, $(F)$ and (\ref{NOVAEQ}) give 
$$
F(.,u_n) \to 0 \quad \mbox{in} \quad L^{q^+}(\mathbb{R}^{N}),  
$$
$$
F(.,u_n) \to 0 \quad \mbox{in} \quad L^{q^-}(\mathbb{R}^{N}),  
$$
$$
f(.,u_n)u_n \to 0 \quad \mbox{in} \quad L^{q^+}(\mathbb{R}^{N})  
$$
and
$$
f(.,u_n)u_n \to 0 \quad \mbox{in} \quad L^{q^-}(\mathbb{R}^{N}).  
$$
The above limits combined with Proposition (\ref{HSL-variable}) imply that 
$$
\int_{\mathbb{R}^N} \frac{F(x,u_n(x))f(y,u_n(y))u_n(y)}{|x-y|^{\lambda(x,y)}} dx dy \to 0.
$$
Now, gathering this limit with $J'(u_n)u_n=o_n(1)$, we find
$$
\int_{\mathbb{R}^N}|\nabla u_n|^{p((x)}+V(x)|u_n(x)|^{p(x)}\, dx \to 0,
$$
from where it follows that
$$
u_n \to 0 \quad \mbox{in} \quad E,
$$
showing that $J(u_n) \to 0$, contradicting again the limit $J(u_n) \rightarrow d >0$.  
\newline

\noindent \textbf{Acknowledgements:}
This work was done while the second author was visiting the  Federal University of Campina Grande. He thanks the hospitality of professor Claudianor Alves and of the other members of the department.

\bigskip


\noindent \textsc{Claudianor O. Alves} \hfill \textsc{Leandro da S. Tavares}\\
Unidade Acad\^{e}mica de Matem\'atica\hfill Unidade Acad\^{e}mica de Matem\'atica\\
Universidade Federal de Campina Grande \hfill Universidade Federal de Campina Grande\\
Campina Grande, PB, Brazil,  \hfill Campina Grande, PB, Brazil,\\
 CEP:58429-900 \hfill CEP:58429-900 \\
\texttt{coalves@mat.ufcg.br} \hfill \texttt{lean.mat.ufcg@gmail.com}

\end{document}